\documentclass[preprint,12pt]{elsarticle}
\usepackage{subfigure}
\usepackage[english]{babel}
\usepackage{graphicx}
\usepackage{extarrows}

\usepackage{lineno}
\usepackage{indentfirst, latexsym, bm, amssymb,amsmath}
\usepackage{mathrsfs}
\usepackage[inline]{enumitem}
\usepackage{amssymb}
\usepackage{pifont}
\usepackage{tikz}
\usepackage{natbib}
\journal{ }
\begin{document}
	%\linenumbers
	\newtheorem{theorem}{Theorem}[section]
	\newtheorem{problem}[theorem]{Problem}
	\newtheorem{corollary}[theorem]{Corollary}
	\newtheorem{observation}[theorem]{Observation}
	\newtheorem{definition}[theorem]{Definition}
	\newtheorem{conjecture}[theorem]{Conjecture}
	\newtheorem{question}[theorem]{Question}
	\newtheorem{lemma}[theorem]{Lemma}
	\newtheorem{proposition}[theorem]{Proposition}
	\newtheorem{example}[theorem]{Example}
	\newenvironment{proof}{\noindent {\bf
			Proof.}}{\hfill $\square$\par\medskip}
	\newcommand{\remark}{\medskip\par\noindent {\bf Remark.~~}}
	\newcommand{\pp}{{\it p.}}
	\newcommand{\de}{\em}

	\newcommand{\1}{{\uppercase\expandafter{\romannumeral1}}}
	\newcommand{\2}{{\uppercase\expandafter{\romannumeral2}}}
	\newcommand{\3}{{\uppercase\expandafter{\romannumeral3}}}
	\newcommand{\4}{{\uppercase\expandafter{\romannumeral4}}}
	
	\begin{frontmatter}
	\title{The  maximum  number  of  cliques  in  disjoint copies of graphs}
%		\tnoteref{tnote1}}
%	\tnotetext[tnote1]{This work is supported by National Natural Science
%			Foundation of China (No. 11871222, 11901554) and Science and Technology Commission
%			of Shanghai Municipality (No. 18dz2271000).}
%			E-mail addresses: gaozhipeng@xidian.edu.cn (Z. Gao),
%			lp-math@snnu.edu.cn (P. Li), chlu@math.ecnu.edu.cn (C. Lu),
%			sunruicaicaicai@163.com (R. Sun),
%			 ltyuan@math.ecnu.edu.cn.
%			}}

		\author[1]{Zhipeng Gao} 
	\ead{gaozhipeng@xidian.edu.cn}

		\author[2]{Ping Li} 
	\ead{lp-math@snnu.edu.cn}

		\author[3]{Changhong Lu}
	\ead{chlu@math.ecnu.edu.cn}

		\author[3]{Rui Sun\corref{cor}} 
	\cortext[cor]{Corresponding author.}
	\ead{sunruicaicaicai@163.com}

		\author[3]{Long-Tu Yuan}
	\ead{ltyuan@math.ecnu.edu.cn}
	
	\address[1] {School of Mathematics and Statistics, Xidian University, Xi'an, 710071, China.}
	
	\address[2] {School of Mathematics and
		Statistics, Shaanxi Normal University, Xi'an, Shaanxi, China. }
		
		\address[3] {School of Mathematical Sciences,
			Shanghai Key Laboratory of PMMP,
			East China Normal University,
			Shanghai 200241, China.}
			
%	\author{ Zhipeng Gao \footnote{School of Mathematics and Statistics, Xidian University, Xi'an, 710071, China. {\tt gaozhipeng@xidian.edu.cn}}, Ping
%		Li\footnote{School of Mathematics and
%			Statistics, Shaanxi Normal University, Xi'an, Shaanxi, China. {\tt
%				lp-math@snnu.edu.cn}},
%	 Changhong Lu, Rui Sun, Long-Tu Yuan \\
%		School of Mathematical Sciences\\
%		Shanghai Key Laboratory of PMMP\\
%		East China Normal University\\
%		Shanghai 200241, China\\
%	}
	\date{}
	
%	\maketitle
	
	\begin{abstract}
	The problem of determining the maximum number of copies of $T$ in an $H$-free graph, for any graphs $T$ and $H$, was considered by Alon and Shikhelman. This is a variant of Tur\'{a}n's classical extremal problem.
We show lower and upper bounds for the maximum number of $s$-cliques in a graph with no disjoint copies of arbitrary graph.	
We also determine the maximum number of $s$-cliques in an $n$-vertex graph that does not contain a disjoint union of $k$ paths of length two when $k=2,3$, or $s\geqslant k+2$, or $n$ is sufficiently large, this partly confirms a conjecture posed by Chen, Yang, Yuan, and Zhang \cite{2024Chen113974}.
		\end{abstract}

	\end{frontmatter}

	\section{Introduction}
	
	 A graph $G$ is said to be $H$-free if it does not contain any subgraph isomorphic to $H$.
	 The classical extremal function $\operatorname{ex}\left(n, H\right)$  denoted as $\operatorname{ex}(n, H)$, is formally defined as the maximum number of edges on an $n$-vertex $H$-free graph.
	This function extends naturally to a scenario where the objective is not merely to maximize the number of edges, but rather to maximize the number of copies of a specified graph $T$ within an $ n $-vertex graph that is $H$-free.
	In accordance with the notation introduced by Alon and Shikhelman \cite{2016Alon146}, the more general function is denoted as $\operatorname{ex}(n, T, H)$, which is commonly referred to the 
	generalized Tur\'{a}n number of $H$.
	Let $P_k$ be the path with $k$ vertices.
	Obviously, $\operatorname{ex}(n, P_2, H) =$ $\operatorname{ex}(n, H)$. 
	We denote 
	the cycle with $\ell$ vertices by $C_\ell$ and the complete graph with $s$ vertices by $K_s$.

	The exploration of such problems dates back to the foundational work of  Erd\H{o}s \cite{1962Erdoes}, who established the values of  $\text{ex}(n, K_s, K_t)$  for any two cliques. After these seminal discoveries, a series of related findings were made, with one of the most notable being the resolution of  $\text{ex}(n, C_5, C_3)$  by Hatami et al. \cite{2013Hatami722}, and Grzesik \cite{2012Grzesik1061}, each working independently. A wealth of additional research has further expanded our understanding of generalized extremal numbers, as evidenced by various studies referenced in the literature  \cite{2008Bollobas4332,2019Gerbner103001,2019Letzter2360,2018Luo219}.
	
	Scholars have studied the generalized Tur\'{a}n numbers for the vertex-disjoint union of graphs. Let $kH$ represent the collection of $k$ disjoint copies of the graph $H$. Gerbner, Methuku, and Vizer \cite{2019Gerbner3130} delved into the function $\operatorname{ex}(n, T, kH)$, where $H$ encompasses a complete graph, a cycle, or a complete bipartite graph. We define $G_{\operatorname{ex}}(n, H)$ as the $H$-free graph consisting of $n$ vertices and $\operatorname{ex}(n, H)$ edges.
% Let $kH$ denote the disjoint copies of the graph $H$.
% 	 Gerbner, Methuku and Vizer \cite{2019Gerbner3130}   investigated the function $\operatorname{ex}(n, T, k  H)$, where $H$ is a complete graph, a cycle or a complete bipartite graph. 
% 	 	We denote $G_{\operatorname{ex}}\left(n, H\right)$ by the $H$-free graph with $\operatorname{ex}\left(n, H\right)$ edges on $n$ vertices.

In this paper, we establish bounds on
$\operatorname{ex}$$(n,K_s, k  H)$ for general $H$. For any graph $H$ and integer $n$, we define $\text{ex}(n, K_0, H) = 1$ and $\text{ex}(n, K_1, H) = n$. Additionally, for any integers $s> k>0$, we establish that $\binom{0}{s} = 1$ and $\binom{k}{s} = 0$. 
% Define ex$(n,K_0,  H)=1$ and ex$(n,K_1,  H)=n$ for any graph $H$ and integer $n$,
% and define $ \binom{0}{s}=1$ for any integer $s$.
	 \begin{theorem}\label{theorem1s}
	 	Let $H$ be a connected graph on $m$ vertices, $k\geqslant 1$  and $n\geqslant km$.
	 	Then $\emph{ex}(n,K_s, k  H)\geqslant$
	 	$$\max \left\{\emph{ex}(n-km+1,K_s,  H)+\binom{km-1}{s},\sum_{i=0}^{s}\emph{ex}(n-k+1,K_i, H)\binom{k-1}{s-i}\right\}.$$
	 \end{theorem}
	 \begin{proof}
The lower bound is derived by counting $s$-cliques in two graphs respectively: $K_{km-1} \cup G_{\operatorname{ex}}(n-km+1, H)$ and $K_{k-1} + G_{\operatorname{ex}}(n-k+1, H)$, both of which exclude $kH$ as a subgraph. Here, $K_{km-1} \cup G_{\operatorname{ex}}(n-km+1, H)$ represents the disjoint union of $K_{km-1}$ and $G_{\operatorname{ex}}(n-km+1, H)$, while $K_{k-1} + G_{\operatorname{ex}}(n-k+1, H)$ denotes the join (adding all possible edges between two graphs) of $K_{k-1}$ and $G_{\operatorname{ex}}(n-k+1, H)$.
 \end{proof}

%	 For $H$ be a subgraph of $G$ by $G-H$ we mean a graph obtained from $G$ by deleting all vertices of $H$ with all incident edges.
%	 Denoted by $\mathcal{N}(G,H)$  the number of (not necessarily induced) copies of $H$ in $G$.

	 \begin{theorem}\label{theorem1r}
	 	Let $H$ be a connected graph on $m\geqslant 3$ vertices and $k\geqslant 1$ be an integer.
	 	Then for $n\geqslant km$ and $ s\geqslant1$ we have
	 	$$\emph{ex}(n,K_s, k  H)\leqslant \sum_{i=0}^{s}\emph{ex}(n-(k-1)m,K_i,  H)\binom{(k-1)m}{s-i}.$$
	 \end{theorem}
	\remark For the case when $s=2$, our results coincide with Theorems 1 and 2 presented in Gorgol \cite{2011Gorgo661}.
 %When  $s=2$, we obtain Theorems 1 and 2 as stated in Gorgol \cite{2011Gorgo661}.

\medskip

	For the notation $F_{k_1, \ldots , k_m}$, we define it as the linear forest comprising the union of paths $P_{k_1} \cup P_{k_2} \cup \ldots \cup P_{k_m}$. For $k_i=2$ over all $i\in [m]$, Wang \cite{2020Wang103057} accurately determined the exact value of $\operatorname{ex}(n,K_s, kP_2)$ in 2020. Zhu, Zhang, and Chen \cite{2021Zhu1437} determined the values of $\operatorname{ex}(n, K_s, F_{k_1, \ldots , k_m})$ for two cases: when $k_1 \leqslant k_2 + \ldots + k_m$ and when $k_1 \geqslant k_2 + \ldots + k_m + 2$, respectively, where $k_1, \ldots , k_m$ are all even numbers and $n$ is sufficiently large.
Expanding further, Zhu and Chen \cite{2022Zhu112997} established the value of $\operatorname{ex}(n, K_s, F_{k_1, \ldots , k_m})$ when $k_1 \geqslant \ldots \geqslant k_m \geqslant 4$, allowing $k_1, \ldots , k_m$ to include odd numbers, again for sufficiently large $n$. Zhang, Wang, and Zhou \cite{2022Zhang} precisely ascertained the values of the extremal function $\operatorname{ex}(n, K_s, \textit{L}_{n,k})$, where $\textit{L}_{n,k}$ represents the family of all linear forests of order $n$ containing exactly $k$ edges. More recently, Chen, Yang, Yuan, and Zhang \cite{2024Chen113974} determined $\operatorname{ex}(n, K_s, F_{k_1, \ldots , k_m})$ under the generalized condition that $k_1\geqslant \ldots \geqslant k_m \geqslant 2$, where $k_1, \ldots , k_m$ are even integers for all value of $n$.
Intriguingly, they also posited the following conjecture:
Let $M_n$ denote the graph consisting of $\lfloor n/2\rfloor$ independent edges and one possible isolated vertex.
For integer $n$, $k$ and $s$, we let
	$$f(n,k,s)=\binom{k-1}{s}+(n-k+1)\binom{k-1}{s-1}+\lfloor (n-k+1)/2\rfloor \binom{k-1}{s-2}.$$
		\begin{conjecture}\label{conjecture1}
		Let $n \geqslant 3 k$ and $s \geqslant 3$. Then 
$$
\operatorname{ex}(n, K_s, k P_3)=\max\left\{ {3k-1 \choose s}, f(n,k,s)\right\}.
$$
Moreover, %each extremal graph is a subgraph of either $K_{3k-1}\cup M_{n-3k+1}$ or $K_{k-1}+M_{n-k+1}$ for $s\leqslant 3k-1$ and any $kP_3$-free graph for $s\geqslant 3k$.
each extremal graph $G$ satisfies the following:
\begin{itemize}
    \item $G$ is a subgraph of $K_{k-1}+M_{n-k+1}$ or it  satisfies that $K_{3k-1}\cup I_{n-3k+1}\subseteq G \subseteq K_{3k-1} \cup M_{n-3k+1}$ for $3\leqslant s\leqslant k+1$,
    \item $G$ satisfies that $K_{3k-1}\cup I_{n-3k+1}\subseteq G \subseteq K_{3k-1} \cup M_{n-3k+1}$ for $k+2\leqslant s\leqslant 3k-1$,
\end{itemize}
%is a subgraph of $K_{k-1}+M_{n-k+1}$ or it  satisfies that $K_{3k-1}\cup I_{n-3k+1}\subseteq G \subseteq K_{3k-1} \cup M_{n-3k+1}$ for $s\leqslant 3k-1$ and any $kP_3$-free graph for $s\geqslant 3k$,
	where $I_{n-3k+1}$ contains $ n - 3k + 1 $ vertices such that no two vertices in $I_{n-3k+1}$ are connected by an edge.
\end{conjecture}

In this paper, we provide a resolution to the conjecture for several cases.

\begin{theorem}\label{theorem3.4}
Conjecture \ref{conjecture1} holds when $k=2,3$.
\end{theorem}

\begin{theorem}\label{theorem3.22}
		Conjecture \ref{conjecture1} holds when  $s\geqslant k+2$.
\end{theorem}

For integers $k,s$, we let
$$g(k,s) = \left( \binom{k-2}{s-2}^{-1}   \max \left\{ \binom{3k-3}{x} : x \in \{s, s-1, s-2\} \right\} \right) (9k-8) + k + 1.$$
\begin{theorem}\label{thm}
	Conjecture \ref{conjecture1} holds when $3\leqslant s \leqslant k$ and $n\geqslant\max\{g(k,s),3k-1\} $.
\end{theorem}
\begin{theorem}\label{theorem3.1}
	Conjecture \ref{conjecture1} holds when $s=k+1$ and  $n\geqslant 6\binom{3k-1}{k}+k-3$.
\end{theorem}

	\section{Proof of Theorem \ref{theorem1r}}\label{section 2.1}
	
For a subgraph $H$ of $G$, we define $G-H$ as the graph derived from $G$ by removing all vertices of $H$ along with all edges incident to these vertices.
Furthermore, let $\mathcal{N}(G,H)$ represent the number of (not necessarily induced) copies of $H$ in $G$.
\medskip

\noindent {\bf Proof of Theorem \ref{theorem1r}.}
We prove by 
double induction on $k$ and $s$. 
Obviously, under our constraints on the binomial coefficient $\binom{k'}{s'}$, Theorem \ref{theorem1r} holds for $k=1$ or $s=1$.
%Theorem \ref{theorem1r} holds for $k=1$ or $s=1$.
Let $k\geqslant2$ and $s\geqslant2$. If the theorem is false, there is an $n$-vertex $k H$-free graph $G$ with $\mathcal{N}(G,K_s)$ $s$-cliques, where 
\begin{equation}\label{new 1}
\mathcal{N}(G,K_s)\geqslant \sum_{i=0}^{s}\mbox{ex}(n-(k-1)m,K_i,  H)\binom{(k-1)m}{s-i}+1.
\end{equation}
Subject to this, let $G$ have the maximum number of edges. Hence, $G$ contains $H$ as a subgraph.
Since $G-H$ is an $(n-m)$-vertex $(k-1)  H$-free graph,
we obtain 
\begin{align}
&\mathcal{N}(G-H,K_s)-\operatorname{ex}(n-m,K_s, (k-1)  H)\nonumber\\
\geqslant &\mathcal{N}(G,K_s) -\sum_{i=0}^{s-1}\mathcal{N}(H,K_{s-i})\mathcal{N}(G-H,K_i) -\operatorname{ex}(n-m,K_s,  (k-1)  H)\nonumber\\
\geqslant&  \mathcal{N}(G,K_s)-\sum_{i=0}^{s}\binom{m}{s-i}\operatorname{ex}(n-m,K_i, (k-1)  H)\nonumber\\
    \geqslant  & \mathcal{N}(G,K_s)	-\sum_{i=0}^{s}\binom{m}{s-i}\sum_{j=0}^{i}    \mbox{ex}(n-(k-1)m,K_j,  H)\binom{(k-2)m}{i-j}    \label{02}\\
=&\mathcal{N}(G,K_s)-\sum_{j=0}^{s} \mbox{ex}(n-(k-1)m,K_j,  H)\sum_{i=j}^{s} \binom{m}{s-i}\binom{(k-2)m}{i-j} \nonumber\\
=&\mathcal{N}(G,K_s)-\sum_{j=0}^{s} \mbox{ex}(n-(k-1)m,K_j,  H)\binom{(k-1)m}{s-j}>0, \nonumber
\end{align}
%which means that $(k-1)  H$ is a subgraph of $G-H$ and therefore $G$ contains $k  H$,
where (\ref{02}) follows from the induction hypothesis.
However, this is impossible by the definition of $\operatorname{ex}(n-m,K_s, (k-1)  H)$.
So the proof is complete.
\hfill$\square$ \medskip

\section{Proof of Theorems \ref{theorem3.4}}\label{section 2.1s}

In this part of the paper, we ascertain the upper bound of number of $s$-cliques subgraphs within an  $n$-vertex graph that excludes  $2   P_3$ and  $3   P_3$ as a subgraph, respectively.
 
	A {\it fan}, denoted by $F_c$, is the graph consisting of $c$ triangles with a common vertex, this is the graph obtained by joining a vertex to all $2c$ vertices of a matching size $c$.
	Call the common vertex in $F_c$ a {\it center} of $F_c$.
	Let $G^-$ be an arbitrary new graph obtained from $G$ by deleting an edge.
	For $S \subseteq V(G)$, the induced subgraph of $G$ by $S$ is denoted by $G[S]$. 
The complement graph of $K_s$ is denoted by $\overline{K_s}$.
With a slight abuse of notation, we say that a graph $H_1$ is incident with another graph $H_2$ if there exists an edge that connects a vertex in $H_1$ to a vertex in $H_2$. 

\medskip

\noindent {\bf Proof of Theorem \ref{theorem3.4} for $k=2$.}
Let $G$ be a $2  P_3$-free $n$-vertex graph with the maximum number of $s$-cliques and let $\widetilde{G}$ be an $n^\prime$-vertex graph derived from $G$ by removing all vertices that do not appear in any copy of $K_s$, where $n\geqslant6$
and $3\leqslant s \leqslant 5$.
It is easy to see that $\mathcal{N}(G,K_s)=\mathcal{N}(\widetilde{G},K_s)	$.
If $n< 3k =6$, then $G=K_n$ and therefore $n'=n$. Hence, we consider the case for $n'\geqslant 5$.

For $s=3$,
since $K_5$ and $F_{\lfloor(n-1)/2\rfloor}$ are $2  P_3$-free graphs,
we obtain
$$\mathcal{N}(G,K_3)	\geqslant\left\{\begin{array}{lll}
	10,
	& \text { for } & 6\leqslant n\leqslant 22;  \\
	\lfloor(n-1)/2\rfloor, & \text { for } & n>22.
\end{array}\right.
$$
Since $G$ does not contain $2P_3$ as a subgraph, it follows that $\widetilde{G}$ contains no two vertex-disjoint copies of $K_3$. Moreover, $\widetilde{G}$ is connected. Given that $\mathcal{N}(\widetilde{G},K_3) \geqslant 10$, it implies that any two copies of $K_3$ within $\widetilde{G}$ have at least one vertex in common.
If any two copies of $K_3$ share two vertices,
	 then $\widetilde{G}$ contains $K_4^{-}$ as a subgraph. 
	 If $n^\prime \geqslant 6$, then
	 $\widetilde{G}$ contains a copy of $K_2+\overline{K_4}$, which result in a   $2P_3$ in $G$, a contradiction. Therefore $n^\prime =5$, which implies that $\widetilde{G}=K_5$.
	If there are two copies of $K_3$ share only one vertex and then
	$F_2$ is a subgraph of $G$. 
	Since $G$ is $2   P_3$-free, 
	each vertex of $G-V(F_2)$ can only be incident with the center of $F_2$.
	According to $\mathcal{N}(\widetilde{G},K_3)$ reaching the maximum value, then $\widetilde{G}=K_5$ or $F_{\lfloor(n-1)/2\rfloor}$. 
	Hence,
		$$
	\operatorname{ex}\left(n,K_3, 2  P_3\right)=\left\{\begin{array}{lll}
		10,
		& \text { for } & 6\leqslant n\leqslant 22;  \\
		\lfloor(n-1)/2\rfloor, & \text { for } & n>22.
	\end{array}\right.
	$$
	Since $G$ contains no $2  P_3$, 
 %any vertex in $G-K_5$ is not incident with vertices in $K_5$ and $G-K_5$ contains no copy of $P_3$ when $n\leqslant 22$. Vertex in $G-F_{\lfloor(n-1)/2\rfloor}$  can only incident with the center of $F_{\lfloor(n-1)/2\rfloor}$  for $n\geqslant 22$.
 $G$ satisfies that $K_{5}\cup I_{n-5}\subseteq G \subseteq K_{5} \cup M_{n-5}$ for $n\leqslant 22$ or it is a subgraph of $K_{1}+M_{n-1}$
 for $n\geqslant 22$, respectively.
	Hence, Theorem \ref{theorem3.4} holds for $k=2$ and $s=3$.
	
	For $s=4$, since $K_5$ is $2  P_3$-free,
	we get $\mathcal{N}(\widetilde{G},K_4)\geqslant 5$.
	There are two copies of $K_4$ sharing three common vertices, otherwise $G$ contains a copy of $2  P_3$, a contradiction.
    Moreover,  any vertex outside these two copies of $K_4$ is not contained in a copy of $K_4$.
	Thus $\widetilde{G}=K_5$, whence $G$ is the disjoint union of $K_5$ and a subgraph of $M_{n-5}$.

For $s=5$, clearly  $\widetilde{G}=K_5$.
Hence, $G$ is the disjoint union of $K_5$ and the subgraph of $M_{n-5}$.
The proof is complete.
\hfill$\square$ \medskip

For the case where $k=3$, we depend on an essential theorem established by Luo \cite{2018Luo219}.
\begin{theorem}[\cite{2018Luo219}\label{theorem0}]  
	Let $n \geqslant k \geqslant 4$ and let $G$ be an $n$-vertex connected graph with no path on $k$ vertices. Let $t=\lfloor(k-2) / 2\rfloor$. Then $\mathcal{N}(G,K_s) \leqslant \max \{f_s(n, k-1,1), f_s(n, k-1, t)\}$, where $f_s(n, k, a)=\binom{k-a}{s}+(n-k+a)\binom{a}{s-1}$.
\end{theorem}

In the following proof, we define two edges as independent if no connecting edges exist between them.

 \medskip
 
\noindent {\bf Proof of Theorem \ref{theorem3.4} for $k=3$.} Since Theorem \ref{theorem3.22} give the case on $s\geqslant 5$, we may assume that $s\in\{3,4\}$ in the subsequent discussion.
%We will prove the theorem by contradictions.
Let $G$ be an $n$-vertex  $3  P_3$-free graph with 
%$\max\{{8 \choose s},f(n,3,s)\}+1$
maximum number of
 $s$-cliques.
Delete the vertices and edges of $G$ which are not in a copy of $K_s$ and denote the obtained graph by $G^\prime$ with $n^\prime\geqslant 9$ vertices (if $n^\prime \leqslant 8$, then $G'=K_8$ and $G$   satisfies that $K_{8}\cup I_{n-8}\subseteq G \subseteq K_{8} \cup M_{n-8}$.).
Since $K_8$ and $K_2+M_{n-2}$ are both $3P_3$-free,
we get $$\mathcal{N}(G,K_s)= \mathcal{N}(G^\prime,K_s)\geqslant \max\left\{{8 \choose s},f(n,3,s)\right\}.$$Assume that $G^\prime$ is disconnected.
Clearly, each component of $G^\prime$ contains a copy of $P_3$.
Thus $G^\prime$ consists of two components $C_1$ with $n_1$ vertices and $C_2$ with $n_2$ vertices, where $n_1+n_2\leqslant n$.
Both of $C_1$ and $C_2$ are $2  P_3$-free.
By Theorem~\ref{theorem3.4} for $k=2$, 
$$\mathcal{N}(G^\prime,K_s)\leqslant \max\left\{{\min\{n_1,5\} \choose s},f(n_1,3,s)\right\}+\max\left\{{\min\{n_2,5\} \choose s},f(n_2,3,s)\right\}.$$
By the maximality of $G^\prime$ and Theorem~\ref{theorem3.4} for $k=2$, we can assume that $C_i$ is either $K_5$ or a fan for $i=1,2$.
Basic calculations show that 
$\mathcal{N}(G,K_s)< \max\{{8 \choose s},f(n,3,s)\}$,
 a contradiction.
Thus $G^\prime$ is connected.

Let $P=P_\ell$ be a longest path in $G^\prime$.
Since $G$ contains no $3  P_3$, we have $\ell\leqslant8$.
If $\ell \leqslant 6$, then $G$ contains no $P_7$.
For $s=3,4$,
by Theorem \ref{theorem0} and basic calculations, we have
$\mathcal{N}(G,K_s)\leqslant \max \{f_s(n, 6,1), f_s(n, 6, 2)\} <\max \{ \binom{8}{s}, f(n,3,s) \}$, a contradiction.
Thus $\ell=7,8$.

\medskip

% {\bf Propositions:} If $l=7$ and let $P_7=x_1x_2x_3x_4x_5x_6x_7$, we can obtain the following facts.

% \begin{enumerate}
%     \item 
% \end{enumerate}
\noindent {\bf Claim.} $\ell=8$.

\medskip

\begin{proof}
If $\ell=7$, then $P_7=x_1x_2\ldots x_7$ be a longest path in $G^\prime$.
Since $G$ is $3P_3$-free, $G^\prime-P_7$ consists of isolated vertices and independent edges.
Suppose that there is a vertex $u$ in $G^\prime-P_7$ which is incident with at least three vertices of $P_7$.
Then it is only incident with $x_2,x_4,x_6$,
since $P_7$ is the longest path in $G^\prime$.
Since $n^\prime\geqslant 9$, there is a vertex $v$ which is incident with $u$ or $P_7$.
For the former, $G'$ contains a path $vux_2 x_3 x_4 x_5 x_6 x_7$ of length 8, a contradicts.
For the latter, 
$G^\prime$ contains a copy of $3  P_3$, a contradiction. 
Thus each vertex of  $G^\prime-P_7$ is incident with at most two vertices of $P_7$. 
According to the definition of $G'$, every vertex and every edge in $G'$ are contained in an $s$-clique of $G$.

If $s = 4$, then, given that $n^{\prime} \geqslant 9$, there necessarily exists an edge $uv$ that lies outside the path $P$, along with two vertices $x_i$ and $x_j$ of $P$, such that the induced subgraph $G^{\prime}[x_i, x_j, u, v]$ forms a $K_4$. However, for every pair of indices $i, j \in [7]$, the graph $G^{\prime}$ contains either a $P_8$ or a $3P_3$, leading to a contradiction.

We next focus on the case where $s=3$. If there are at least two isolated vertices in $G^\prime-P_7$.
Choose isolated vertices $w_1$, $w_2$ of $G^\prime-P_7$. Note that $w_i$ has exactly two neighbors in $P$ to form a triangle.
Hence, by symmetry, we only consider the following four cases of $N_P(w_1)$, where $N_P(w_1)$ is the set of neighbors of $w_1$ in $P$.

\medskip

{\bf Case 1.} $N_P(w_1)=\{x_2,x_4\}$.
In this case, regardless of which two vertices in $P_7$ are adjacent to $w_2$, $G'[\{w_1, w_2\} \cup V(P_7)]$ contains 3 disjoint copies of $P_3$, which is a contradiction.

\medskip

{\bf Case 2.} $N_P(w_1)=\{x_2,x_5\}$.
If $w_2$ is adjacent to $x_2$ and $x_4$, then we conclude by Case 1. Since $G$ contains no $3 P_3$, $w_2$ can only be adjacent to $x_2,x_5$ and $G'-\{x_2,x_5\}$ contains no $P_3$.
Thus  $G^\prime-\{x_2,x_5\}$ consists of isolated vertices and independent edges. Consequently, $G'\subset K_2+M_{n-2}$ (where $K_2=x_2x_5$), implying that $\mathcal{N}(G',K_3)<f(n,3,3)$, a contradiction.

\medskip

{\bf Case 3.} $N_P(w_1)=\{x_2,x_6\}$.
Then any other isolated vertex of $G^\prime-P_7$ cannot be incident with $P_7$ as $G^\prime$ is $3  P_3$-free.
Thus we get a contradiction.

\medskip

{\bf Case 4.} $N_P(w_1)=\{x_3,x_5\}$.
Since $G^\prime$ is $3  P_3$-free, any other isolated vertex of $G^\prime-P_7$ can only be adjacent to $x_3$ and $x_5$ and $G'-\{x_3,x_5\}$ contains no $P_3$, implying $G^\prime-\{x_3,x_5\}$ consists of isolated vertices and independent edges
as Case 2, a contradiction.

Now, we assume that there exists at most one isolated vertex, implying the presence of an edge $uv$ within $G' - P_7$ as $n' \geqslant 9$.  
Since $G^\prime$ is $3  P_3$-free and $P_7$ is the longest path, $uv$ can only be incident with $x_3$ or $x_5$.  
Since $uv\in E(G^\prime)$, $uv$ is contained in some triangle. Without loss of generality, let $uvx_3u$ form such a triangle. We first deduce that $x_1x_2$ is only incident with $x_3$ in $G'$.
If both $x_3$ and $x_5$ are incident with $uv$, then it implies the existence of a path of length 8, which contradicts our initial assumption.
Therefore, $uv$ is only incident with $x_3$ in $P_7$. On the one side, given that $P_7$ is the longest path in $G^\prime$, it follows that $x_1x_2$ can only be incident with $x_3$ within the set $\{x_3,\ldots,x_7\}$. On the other side, $x_1x_2$ is an edge of some triangle by definition of $G'$, implying that $x_1x_2x_3x_1$ forms one such triangle.   
Let $w\in G-P_7$, then $w$ can't be the neighbor of $x_1$ or $x_2$. Otherwise, a $3P_3$ will occur.

If there is another edge incident with $x_5$, then following the same idea, we deduce that $x_6x_7$ is only incident with $x_5$ in $G'$.
Under this condition, we show that $x_4$ has no neighbor except $\{x_3,x_5\}$ in $G'$. Indeed, $x_4$ has no neighbor in $P_7-\{x_3,x_5\}$ from the above discussion. If $x_4$ is incident with an edge outside $P_7$, then we can easily obtain a $3P_3$. If $x_4$ is incident with a vertex $w$, we can obtain a longer path as $wx_4x_iw$ (where $i=3$ or $i=5$) forms a triangle. Both cases lead to contradictions. Consequently, we can infer that $G'-\{x_3,x_5\}$ consists of $M_{n-2}$. Then $G'\subset K_2+M_{n-2}$, implying $\mathcal{N}(G',K_3)<f(n,3,3)$, a contradiction. 

Then all edges outside $P_7$ are only incident with $x_3$. Since there is at most one isolated vertex in $G'-P_7$, 
then this possible unique vertex (denoted by $w$) cannot be incident with $x_1x_2$ as $G'$ is $3P_3$-free, implying that $G^\prime$ is the subgraph of $(K_5\cup M_{\lfloor (n-6)/2\rfloor})+K_1$, where $K_1= x_3$, $K_5$ consists of vertices $w, x_4, x_5, x_6, x_7$ and $M_{\lfloor (n-6)/2\rfloor}$ consists of the remaining vertices. Thus $\mathcal{N}(G^\prime,K_3)\leqslant{6 \choose 3} + \lfloor (n-6)/2\rfloor< \max\{{8 \choose 3},f(n,3,3)\}$, a contradiction.
\end{proof}

Let $P=x_1x_2\ldots x_8$ be a longest path in $G^\prime$.
Then  $G^\prime-P$ comprises isolated vertices and independent edges. Since $G^\prime$ is $3P_3$-free, every vertex $v$ in $G^\prime - P$ can only have its neighbors in $\{x_3, x_6\}$. Furthermore, by definition of $G'$, $x_3$ and $x_6$ must be neighbors of $v$ to form an $s$-clique. 
Now, we endeavor to prove that, $G' = K_2 + M_{n-2}$ when $ n' \geqslant 9$. 

If $G' - \{x_3, x_6\}$ consists of isolated vertices and independent edges, then $G'$ is a subgraph of $K_2 + M_{n-2}$ (where $K_2=x_3x_6$). This implies that $\mathcal{N}(G^\prime,K_s) \leqslant f(n^\prime,3,s)\leqslant f(n,3,s)$ for $s=3,4$. Furthermore, the equality $\mathcal{N}(G^\prime,K_s)=f(n,3,s)$ holds
if and only if $G=G^\prime=K_2+M_{n-2}$.
If $G'-\{x_3,x_6\}$ contains a $P_3$, say $uvw$, as a subgraph, then we conclude that $uvw \subset G'[\{x_1, x_2, x_4, x_5, x_7, x_8\}]$ as no vertex outside $P$ has neighbor in $\{x_1, x_2, x_4, x_5, x_7, x_8\}$. Therefore, there is an edge that connects some pair of $\{x_1x_2, x_4x_5, x_7x_8\}$.  
%If $G' \neq K_2 + M_{n-2}$, then as $\mathcal{N}(G', K_s) \geqslant\max\{{8 \choose s},f(n,3,s)\}$, $G' $ is not a subgraph of $K_2 + M_{n-2}$ and $G' -\{x_3, x_6\}$ contains a $P_3$%. If $G' - \{x_3, x_6\}$ contains $2P_3 $, then given $ n' \geqslant 9$, $G' - \{x_1, x_2, x_4, x_5, x_7, x_8\}$ contains a third $P_3$, which is a contradiction. Therefore, $G' -\{x_3, x_6\}$ contains only one $P_3$, denoted as $uvw$. Since each vertex in $ G' - P $ is adjacent only to $x_3$ or $ x_6 $ in $P $, it follows that $\{u, v, w \}\subset \{x_1, x_2, x_4, x_5, x_7, x_8\} $.

If $uvw\subset G[\{x_1, x_2, x_7, x_8\}]$, then 
there is exactly one isolated vertex in $G'-P$. Indeed, if $u_1$ and $u_2$ are vertices in $G'-P$, then they must either be adjacent to the same vertex in $\{x_3,x_6\}$ or to different vertices in $\{x_3,x_6\}$. In either case, $G' -\{u, v, w\}$ would contain a $ 2P_3$, leading to a contradiction. 
Therefore, $n' = 9 $, and we denote the final vertex as $x_9$. Then the only edge that connects $x_1x_2$ and $x_7x_8$ is $x_2x_7$, otherwise, there will be a $P_9$ including $x_9$.
Since $G[V(P)]\subseteq K_8-\{x_1x_8,x_1x_7, x_2x_8\}$, we obtain that $\mathcal{N}(G^\prime,K_s)< {8 \choose s}-1+{3 \choose s}\leqslant \max\{{8 \choose s},f(n,3,s)\}$ for $s=3,4$. Here ${8 \choose s}-1$ represents the maximum number of $s$-cliques in $G[V(P)]$ and ${3 \choose s}$ signifies the number of $s$-cliques in $G'[x_3,x_6,x_9]$. This contradicts our initial assumption. Consequently, there is no edge connecting $x_1x_2$ and $x_7x_8$ and $uvw\subset G[\{x_1, x_2, x_4, x_5\}]$ or $uvw\subset G[\{x_4, x_5, x_7, x_8\}]$. By symmetry, we can assume that $uvw\subset G[\{x_1, x_2, x_4, x_5\}]$. Since $G$ is $3P_3$-free, there is at most one isolated vertex (denoted as $v'$) in $G' - P$ that is adjacent to $x_3$. Furthermore, since each vertex in $G' - P$ is only adjacent to $x_3$ or $x_6$ in $P$, the remaining $n' - 9$ vertices in $G' - P$ must be solely adjacent to $x_6$. 
Therefore, $\mathcal{N}(G^\prime,K_s)<{6 \choose s}+{6 \choose s}+\lfloor(n'-8)/2\rfloor{3 \choose s}+{3 \choose s}< \max\{{8 \choose s},f(n,3,s)\}$for $s=3,4$. Here, ${6 \choose s}$ represents the number of $s$-cliques in $G[\{x_1, x_2, x_3, x_4, x_5, x_6\}]$, 
another ${6 \choose s}$ corresponds to the number of $s$-cliques in $G[\{x_3, x_4, x_5, x_6, x_7, x_8\}]$,
$\lfloor(n'-8)/2\rfloor{3 \choose s}$ represents the maximum number of $s$-cliques in $G' -(V(P)\setminus\{x_6\})$ and the remaining ${3 \choose s}$ represents the number of $s$-cliques formed by the vertices in $\{x_3,x_6,v'\}$, a contradiction.

Above all,  for $n' \geqslant 9$, we have $G' = K_2 + M_{n-2}$, which implies that $\mathcal{N}(G', K_s) = \max\left\{{8 \choose s}, f(n, 3, s)\right\}$ for $s = 3, 4$. This completes the proof.
\hfill$\square$ \medskip

\section{Proof of Theorems \ref{theorem3.22}}\label{section 3}
 Let $G$ be a $k   P_3$-free graph on $n$ vertices with the maximum number of $s$-cliques and subject to which, let $G$ contain as many edges as possible. Therefore, $G$ contains a $(k-1)P_3$ as a subgraph.
	Let $H = (k-1)  P_3$ be the $k-1$ disjoint paths $x_1y_1z_1, \ldots, x_{k-1} y_{k-1} z_{k-1}$ of $G$ 
	such that $G' = G-V(H)$ has the maximum number of edges.
	Before proving the main result, we present a technical lemma.
	
\begin{lemma}\label{lemma2.1} 
	Let $u_1 v_1,  u_2 v_2$ be edges and  $w_1, w_2$ be isolated vertices in $G'$.
\begin{itemize}
	\item (a) If   $w_1, w_2$ are adjacent to at least one vertex in $\left\{x_1, y_1, z_1\right\}$, then they are only adjacent to $y_1$,
	\item (b) If $u_1$ is adjacent to at least two vertices in $\left\{x_1, y_1, z_1\right\}$, 
	then $u_2,v_2$ are not adjacent to $x_1, y_1, z_1$ and at most one of $w_1,w_2$ is adjacent to  $x_1, y_1, z_1$. 
\end{itemize}
\end{lemma}
	\noindent {\bf Proof.}
	(a) If $w_{1}$ is adjacent to $x_1$ (or $z_1$), then $w_{2}$ can not be adjacent to vertices in $\left\{x_1, y_1, z_1\right\}$,
	as otherwise $G[\{w_{1}, w_{2}, x_1, y_1, z_1\}]$ contains disjoint union of $P_3$ and an edge, contradicting our choice of $G'$.
	Thus  $w_{1}$ and $w_{2}$ can be only adjacent to $ y_1$.

	(b) If $u_1$ is adjacent to at least two vertices in $\left\{x_1, y_1, z_1\right\}$ and there are some edges between $\{u_2, v_2\}$ and $\{ x_1, y_1, z_1\}$, then $u_2$, $v_2$ with one vertex in $ \{ x_1, y_1, z_1\}$ (which adjacent to $u_2$ or $v_2$) form one copy of $P_3$ and $v_1$,$u_1$ with another vertex in $ \{ x_1, y_1, z_1\}$ (which adjacent to $u_1$) form the second copy of $P_3$,
	 contradicting $G$ be a  $k   P_3$-free graph.
	 Hence, $u_2,v_2$ are not adjacent to $x_1, y_1, z_1$.
	If $w_1,w_2$ are adjacent to $x_1, y_1, z_1$, 
	then by (a), we get they are only adjacent to $y_1$.
	Therefore, $G[\{w_1, w_2, u_1, v_1, x_1, y_1, z_1\}]$ contains  $2   P_3$, 
	 contradicting $G$ contains no  $k   P_3$.
	\hfill$\square$ 

\medskip

\noindent {\bf Proof of Theorem \ref{theorem3.22}.}
Let $k\geqslant2$, $n\geqslant 3k$ and $s \geqslant k+2$.
We will prove the theorem by induction on $k$.
The case $k=2$ is covered by Theorem \ref{theorem3.4} and the case $s\geqslant 3k$ holds trivially.
Let $k\geqslant 3$ and $s\leqslant 3k-1$.
Obviously, $\mathcal{N}(G,K_s)\geqslant \binom{3k-1}{s}$ as $K_{3k-1}\cup M_{n-3k+1}$ is $kP_3$-free.
%\textcolor{red}{By induction on $k$, $G$ contains $(k-1)   P_3$ as a subgraph.}
Let $X\subseteq V(G)\setminus V(H)$ be the vertices in a copy of $K_s$.

\medskip

\noindent {\bf Claim.} $G[X]=K_2$.

\medskip

\begin{proof}
Clearly, $G[X]$ consists of isolated vertices and independent edges.
If $|X|\leqslant 1$, then the number of vertices that appear in a copy of $K_s$ is at most $|V(H)\cup X|\leqslant 3k-2$, implying 
$\mathcal{N}(G,K_s)< {3k-1 \choose s}$. Therefore, $|X|\geqslant 2$.
Suppose that $G[X]$ does not contain an edge and let $u\in X$.
Since $s-1\geqslant k+1$, there is a copy of $K_s$ containing $u$ and two vertices of $x_iy_iz_i$ for some $i\in [k-1]$. Without loss of generality, let $i=1$.
Then by Lemma~\ref{lemma2.1}(a), each isolated vertex of $G'-\{u\}$ is not incident with $x_1y_1z_1$. Therefore, any $s$-clique that involves at least one vertices in $\{x_1,y_1,z_1\}$ is contained in $G[V(H)\cup \{u\}]$. It follows that there are at most $3k - 5$ vertices outside $\{x_1,y_1,z_1\}$ that appear in such $s$-clique.
Since $G-x_1y_1z_1$ is $(k-1)P_3$-free, by inductive hypothesis, $\mathcal{N}(G-x_1y_1z_1,K_s)\leqslant \binom{3k-4}{s}$. Moreover, we further obtain the inequality
\begin{align*}
\mathcal{N}(G,K_s)\leqslant  &~ \mathcal{N}(G-x_1y_1z_1,K_s)  +  \sum_{p\in\{x_1, y_1, z_1\}}\mathcal{N}_p(G-x_1y_1z_1,K_{s-1}) \\
 &+ \sum_{pq\in\{x_1y_1, y_1z_1, x_1z_1\}}\mathcal{N}_{pq}(G-x_1y_1z_1,K_{s-2})   + \mathcal{N}^\ast(G-x_1y_1z_1,K_{s-3}) \\
              \leqslant&\binom{3k-4}{s}+3\binom{3k-5}{s-1}+3\binom{3k-5}{s-2}+\binom{3k-5}{s-3} <\binom{3k-1}{s},
\end{align*} 
here, $\mathcal{N}_p(G-x_1y_1z_1,K_{s-1})$ denotes the number of copies of $K_{s-1}$ in $G-x_1y_1z_1$ that include a vertex $p\in\{x_1, y_1, z_1\}$, $\mathcal{N}_{pq}(G-x_1y_1z_1,K_{s-1})$ represents the number of copies of $K_{s-2}$ in $G-x_1y_1z_1$ that encompass an edge $pq\in\{x_1y_1, y_1z_1, x_1z_1\}$, and
$\mathcal{N}^\ast(G-x_1y_1z_1,K_{s-3})$ signifies the number of copies of $K_{s-3}$ in $G-x_1y_1z_1$ that incorporate the triangle $x_1y_1z_1x_1$. However, it contradicts our initial assumption.

Thus we assume that there is an edge $uv$ in $G[X]$.
Since $s-2\geqslant k$, by symmetry, there is a copy of $K_s$ containing $u$ and two vertices of $x_1y_1z_1$. Without loss of generality, we may assume that $u$ is adjacent to $x_1$. If $|X|\geqslant 3$, then by Lemma~\ref{lemma2.1}(b), at most one isolated vertex of $G'-\{u,v\}$ is incident with $x_1y_1z_1$. If such an isolated vertex exists, we denote this vertex as $w$. Since $G$ is $kP_3$-free, $w$ can only be adjacent to $x_1$, $u$ is only adjacent to $x_1$ and $y_1$, and $v$ cannot be adjacent to any vertex of $x_1,y_1,z_1$.
Hence, any $s$-clique with at least one vertex from $\{x_1, y_1, z_1\}$ lies within $G[V(H) \cup \{u, w\}]$. If $x_1$ is in the clique, it has at most $3k - 4$ vertices outside $\{x_1, y_1, z_1\}$. If not, it has at most $3k - 5$ such vertices.
%Therefore, if $K_s \cap \{u, v, w\} \neq \emptyset$, then regardless of whether $w$ exists, the intersection of $K_s$ and $\{u, v, w\}$ has at most two vertices. 
%in $G - x_1y_1z_1$, there are at most $3k - 5$ vertices that appear in the $s$-clique with vertices from $\{y_1, z_1\}$, and at most $3k - 4$ vertices that form a copy of $K_s$ with $x_1$.
By induction, we have $\mathcal{N}(G-x_1y_1z_1,K_s)\leqslant \binom{3k-4}{s}$. We further obtain that
\begin{align*}
\mathcal{N}(G,K_s)&\leqslant   \mathcal{N}(G-x_1y_1z_1,K_s)  + \sum_{p\in\{x_1, y_1, z_1\}}\mathcal{N}_p(G-x_1y_1z_1,K_{s-1}) \\
 &+ \sum_{pq\in\{x_1y_1, y_1z_1, x_1z_1\}}\mathcal{N}_{pq}(G-x_1y_1z_1,K_{s-2})   + \mathcal{N}^\ast(G-x_1y_1z_1,K_{s-3}) \\
%&\leqslant   \mathcal{N}(G-x_1y_1z_1,K_s)  +3\binom{3(k-2)+1}{s-1}+3\binom{3(k-2)+1}{s-2}+\binom{3(k-2)+1}{s-3} \\
              &< \binom{3k-4}{s}+3\binom{3k-4}{s-1}+3\binom{3k-4}{s-2}+\binom{3k-4}{s-3} =\binom{3k-1}{s},
\end{align*}
a contradiction. If no such isolated vertex exists, then by Lemma~\ref{lemma2.1}(b), only $uv$ of $G[X]$ is incident with $x_1y_1z_1$.
By induction, we have $\mathcal{N}(G-x_1y_1z_1,K_s)\leqslant \binom{3k-4}{s}$. We further obtain that
\begin{align*}
\mathcal{N}(G,K_s)  %\mathcal{N}(G-x_1y_1z_1,K_s)  + \sum_{p\in\{x_1, y_1, z_1\}}\mathcal{N}_p(G-x_1y_1z_1,K_{s-1}) \\
 %&+ \sum_{pq\in\{x_1y_1, y_1z_1, x_1z_1\}}\mathcal{N}_{pq}(G-x_1y_1z_1,K_{s-2})   + \mathcal{N}^\ast(G-x_1y_1z_1,K_{s-3}) \\
              \leqslant \binom{3k-4}{s}+3\binom{3k-4}{s-1}+3\binom{3k-4}{s-2}+\binom{3k-4}{s-3} =\binom{3k-1}{s},
\end{align*}
$\mathcal{N}(G,K_s)=\binom{3k-1}{s}$ if and only if $G-x_1y_1z_1 = K_{3k-4}$ and $G$ contains $K_{3k-1}$ as a subgraph. In this case, $G[X]=K_2$, otherwise $G$ contains $P_{3k}$ as a subgraph, a contradiction.
Therefore, $|X|=2$.
\end{proof}
By the above claim,   $\operatorname{ex}\left(n,K_s, k  P_3\right)= \binom{3k-1}{s}$ and each  extremal graph $F$ satisfies that  $K_{3k-1}\cup I_{n-3k+1}\subseteq F \subseteq K_{3k-1} \cup M_{n-3k+1}$.
The proof is complete.
\hfill$\square$ \medskip

\section{Proof of Theorem \ref{thm} and Theorem \ref{theorem3.1}}\label{section 2}

	%Let $G$ be an $n$-vertex $k  P_3$-free graph with $(k-1)  P_3$ as a subgraph. Furthermore, let $\mathcal{Q}$ denote the set of $k-1$ disjoint $P_3$ subgraphs (specifically, $Q_1=x_1y_1z_1, \ldots, Q_{k-1}=x_{k-1}y_{k-1}z_{k-1}$) in $G$ that maximizes the quantity $|\{i: G[V(Q_i)] \text{ is a triangle}\}|$. Among all such sets $\mathcal{Q}$ that satisfy this condition, let $G' =G- \bigcup_{i=1}^{k-1}V(Q_i)$ be the graph with the maximum number of edges. Note that $G^{\prime}$ consists of isolated vertices and independent edges. For simplicity, denote by $e(G)$ the number of edges of $G$, and let $N_{G'}(Q_i)$ be the set of neighbors of  vertices  of  $V(Q_i)$ in $G'$. 
	Let $G$ be an $n$-vertex $k  P_3$-free graph with $(k-1)  P_3$ as a subgraph. 
Furthermore, let $\mathcal{Q}$ represent the set of $k-1$ disjoint $P_3$ subgraphs (specifically, $Q_1=x_1y_1z_1, \ldots, Q_{k-1}=x_{k-1}y_{k-1}z_{k-1}$) in $G$ that maximizes the number of indices $i$ for which $G[V(Q_i)]$ forms a triangle. Among all such sets $\mathcal{Q}$ that meet this criterion, let $G' = G - \bigcup_{i=1}^{k-1} V(Q_i)$ be the graph with the maximum number of edges. Note that $G'$ consists of isolated vertices and independent edges.
For simplicity, let $e(G)$ represent the number of edges in $G$.

We divide $\mathcal{Q}$ into two sets  $\mathcal{Q}_1$ and $\mathcal{Q}_2$,
where $\mathcal{Q}_1$ consists of $Q_i\in \mathcal{Q}$ that incident with at least four components of $G'$,	  
and $\mathcal{Q}_2$ consist of $Q_i\in \mathcal{Q}$ which incident with at most three components of $G'$. Define $N_{G'}(Q_i)$ as the set of neighbors of the vertices in $V(Q_i)$ within $G'$, and let $V(\mathcal{Q}_i)$ be the union of the vertex sets of all $Q$ in $\mathcal{Q}_i$ with $i=1,2$. 

 To prove Theorem \ref{thm}, we proceed by calculating the maximum possible number of $K_s$ that either contain at least one vertex from $\mathcal{Q}_2$ (as established in Theorem \ref{new2}) or are entirely contained within the subgraph induced by the vertices of $\mathcal{Q}_1$ and $G'$ (i.e., $G[V(\mathcal{Q}_1) \cup V(G')]$), as detailed in Theorem \ref{new1}. For convenience, we introduce the notation $N[U]$ to represent the set of vertices $\{v\in N[u]: u\in U\}$, where $N[u]$ denotes the closed neighborhood of vertex $u$, i.e., $N[u]= N(u) \cup \{u\}$. 

An easy but helpful proposition can be obtained:
\begin{proposition}\label{prok}
Let $S$ be the set of any $t\leqslant k-1$ disjoint copies of $P_3$ of $\mathcal{Q}$. Then $G[\bigcup_{Q\in S}N_{G'} [Q]]$ is $(t+1)P_3$-free.
\end{proposition}

\noindent {\bf Proof.}
    For convenience, let $S=\{Q_1,\ldots, Q_t\}$. If there is a $(t+1)P_3$ contained in $G[\bigcup_{i=1}^t N_{G'}[Q_i]]$, then combining it with $Q_{t+1},\ldots,Q_{k-1}$ are $k$ disjoint copies of $P_3$ in $G$, which is impossible.
\hfill$\square$ 

 Before proving Theorem \ref{thm}, we present several useful lemmas.
\begin{lemma}\label{cor-1}
For each $Q_i\in \mathcal{Q}_1$, one of the following statements holds.
\begin{itemize}
	\item (a) All distinct components incident with $Q_i$ are incident with the same vertex in $V(Q_i)$.
  \item (b) $G[V(Q_i)]$ is $P_3$ and there are two vertices $w\in\{x_i,z_i\}$ and $v\in N_{G'}(Q_i)$, such that $v$ is an isolated vertex in $G'$, $v$ is adjacent to only $w$ in $Q_i$ and all vertices in $N_{G'}(Q_i)-\{v\}$ are adjacent to $y_i$.
\end{itemize}
Moreover, if (b) holds, then $G'$ contains at least one edge.
\end{lemma}

\noindent {\bf Proof.}
Since $Q_i\in \mathcal{Q}_1$, it follows that $Q_i$ is incident with
$p\geqslant 4$ components of $G'$ with $v_1, v_2,\ldots v_p \in N_{G'}(Q_i)$ originating from these components, respectively.
For convenience, let $S=\{v_1, v_2,\ldots v_p\}$.
We only consider $Q_i$ when the statement (a) does not hold, which indicates that $S$ has different neighbors in $V(Q_i)$.

If $G[V(Q_i)]$ is a triangle, then since $S$ has different neighbors in $V(Q_i)$, 
we can easy see that $G[V(Q_i) \cup S]$ contains a $2P_3$, contradicting Proposition \ref{prok}.
Hence, $G[V(Q_i)]$ is a $P_3$.
It is clear that both $x_i$ and $z_i$ have at most one neighbor in $V(G')$ and this possible neighbor is an isolated vertex in $G'$. Otherwise, $G[V(Q_i) \cup S]$ contains a $2P_3$.
Without loss of generality, let $w_1,w_2\in V(G')$ be the neighbors of $x_i$ and $z_i$ (if they exist), respectively. Note that $y_i$ is incident with at least $p-2\geqslant 2$ components of $G'$, denote as $D_1,D_2,\ldots,D_{p-2}$, each of which is distinct from $w_1$ or $w_2$ (note that each $D_i$ represents either an isolated vertex or an edge).
If both $w_1,w_2$ exist and $w_1=w_2$, then $G[V(Q_i) \cup S]$ contains a $2P_3$, a contradiction.
If both $w_1,w_2$ exist and $w_1\neq w_2$, then $Q_1'=Q_1,\ldots,Q_{i-1}'=Q_{i-1},Q'_{i}=y_iz_iw_2,Q'_{i+1}=Q_{i+1},\ldots Q'_{k-1}=Q_{k-1}$ are $k-1$ vertex disjoint $P_3$. However, $|\{i:G[V(Q'_i)]\ \text{is a triangle}\}|$ is not decrease but  $G-\bigcup_{i=1}^{k-1}V(Q'_i)$ has more edges, contradicting the choice of $\mathcal{Q}$.
Hence, only one of $w_1,w_2$  exists, the statement (b) holds.
\hfill$\square$

\medskip

According to Lemma \ref{cor-1}, each $Q_i \in \mathcal{Q}_1$ has a unique vertex, termed the {\em dominating vertex}, that is incident with all but at most one vertices in $N_{G'}(Q_i)$. The remaining two vertices of $Q_i$ are referred to as {\em residual vertices}. 
Each isolated vertex in $G'$ joining some residual vertex is call a {\em special vertex}.  The following result is straightforward from Lemma \ref{cor-1}.

\begin{corollary}\label{cor-11}
Each dominating vertex incident with at least three components of $G'$. Moreover, if the dominating vertex is not  $y_i\in V(Q_i)$, then it incident with at least four components of $G'$.  
\end{corollary}

We define $L$ as the set comprising all residual vertices of $Q_i \in \mathcal{Q}_1$, noting that $|L| = 2|\mathcal{Q}_1|$. Furthermore, we partition $L$ into $L_1$ and $L_2$: $L_1$ consists of the residual vertices of $Q_i \in \mathcal{Q}_1$ where $G[V(Q_i)]$ forms a $P_3$, and $L_2$ consists of the residual vertices of $Q_i \in \mathcal{Q}_1$ where $G[V(Q_i)]$ forms a triangle.
\begin{lemma}\label{lem-special} 
Each special vertex in $G'$ is adjacent to exact one vertex in $L_1$.
\end{lemma}
\noindent {\bf Proof.}
By Lemma \ref{cor-1}(b), each special vertex $w$ is adjacent to at least one vertex in $L_1$.
Suppose to the contrary that $w$ is adjacent to two vertices $a,b\in L_1$ with $a\in V(Q_i)$ and $b\in V(Q_j)$.
By Lemma \ref{cor-1}(b), $i\neq j$ and $y_i,y_j$ are dominating vertices.
Without loss of generality, assume that $a=x_i$ and $b=x_j$.
By Corollary \ref{cor-11},  both $y_i$ and $y_j$ incident with at least three components of $G'$ and then there are two vertices $w_1,w_2\in V(G'-\{w\})$ such that $y_iw_1,y_jw_2\in E(G)$.
Therefore, $z_iy_iw_1,z_jy_jw_2,x_iwx_j$ form a $3P_3$, which contradicts Proposition \ref{prok}.
\hfill$\square$ 

For simplicity, we use the notation $\langle xyz \rangle$ to represent a $P_3$ on the vertices $x$, $y$, and $z$.

\begin{lemma}\label{lem-in} 
$G[L]$ is a $P_3$-free graph.
	%The following results hold.
	%\begin{itemize}
	%\item (a) $G[L]$ is a $P_3$-free graph.
%\item (b) For each edge $ab\in G[L_1]$, if $ab$ is an edge of a triangle in $G[V(\bigcup_{Q \in \mathcal{Q}_1} Q)]$, then the third vertex of the triangle is the dominating vertex of some $Q_k$, where $G[V(Q_k)]$ is a triangle.		
%	\end{itemize}
\end{lemma}
\noindent {\bf Proof.}
 Suppose, instead, there exists a vertex $\alpha \in V(Q_i)$ of $G[L]$ with degree at least two (denoted by $\beta, \gamma$ the neighbors of $\alpha$).
Note $\beta\alpha\gamma$ forms a copy of $P_3$ in $G[L]$, and at most two of $\{\alpha, \beta, \gamma\}$ belong to the same $Q \in \mathcal{Q}_1$. Assume $\beta, \gamma$ are from $Q_j, Q_k$ in $\mathcal{Q}_1$ (where $\{i, j, k\}$ is a multiset), with $\alpha', \beta', \gamma'$ as their dominating vertices, and $\alpha_1, \beta_1, \gamma_1$ as their left vertices, respectively. 
If $G[V(Q_i)]$ is a triangle (without loss of generality, assume that $i\neq j$), then by Lemma \ref{cor-1}, $\alpha'$ incident with  at least four components of $G'$ and $\beta'$ incident with at least three components of $G'$, implying there are four distinct vertex $a,b,c,d$ such that $a\alpha',b\alpha',c\beta',d\beta'\in E(G)$. It is clear that $\beta\alpha\alpha_1,a\alpha' b,c\beta' d$ form a copy of $3P_3$ in the graph induced by $V(Q_i)\cup V(Q_j)\cup V(G')$, contradicting Proposition \ref{prok}. Thus, $G[V(Q_i)]$ is a $P_3$.

If there exists a vertex among $\{\beta, \gamma\}$ that belongs to $\{y_j, y_k\}$, without loss of generality, let $\beta=y_j$. Then $y_j$ is not the dominating vertex of $Q_j$, implying $\beta'$ incident with at least four components of $G'$ by Corollary \ref{cor-11}, and at least one of $\alpha$ or $\gamma$ is not in $Q_j$. Assume that $\alpha \notin Q_j$, implying $i \neq j$.
Since the dominating vertex $\alpha'$ of $Q_i$ incident with at least three components of $G'$ by Corollary \ref{cor-11}, there are four distinct vertices $a,b,c,d\in V(G')$ such that $a\alpha',b\alpha',c\beta',d\beta'\in E(G)$.
Hence, $ a\alpha' b,c\beta' d,\alpha\beta\beta_1$ form a $3P_3$, which contradicts Proposition \ref{prok}. 
If $\alpha=y_i$, then based on a proof similar to $\beta = y_j$, a contradiction is obtained.
Therefore, there is no vertex in $\{\alpha,\beta,\gamma\}$ belongs to $\{y_i,y_j,y_k\}$.

If $\beta$ or $\gamma$ is in $Q_i$, then either $y_i$ belongs to $\{\alpha,\beta,\gamma\}$ or $G[V(Q_i)]$ is a triangle, which is impossible.
Hence, each neighbor of $\alpha$ in $L$ is not in $Q_i$. 
Without loss of generality, set $\alpha = x_i$.
If $j \neq k$, since each of the dominating vertices $\alpha',\beta',\gamma'$ has at least three neighbors in $G'$, there exist $a', b', c'$ from distinct components of $G'$  with $a'\alpha', b'\beta', c'\gamma' \in E(G)$. 
Then $\beta\alpha\gamma,\langle\alpha_1y_ia'\rangle,\langle\beta_1y_jb'\rangle,
\langle\gamma_1y_kc'\rangle$ form a $4P_3$ in $G[V(Q_i) \cup V(Q_j) \cup V(Q_k) \cup \{a', b', c'\}]$, contradicting Proposition \ref{prok}. 
If $j = k$, then $\beta' = \gamma' = y_j$, $\{\beta, \gamma\} = \{x_j, z_j\}$, and there exist $a', b', c' \in V(G')$ with $a'\alpha', b'\beta', c'\beta' \in E(G)$. Clearly, $a'y_i\alpha_1,b'y_jc', \beta\alpha\gamma$ form a $3P_3$ in $G[V(Q_i) \cup V(Q_j) \cup \{a', b', c'\}]$, which is impossible.
\hfill$\square$ 

\medskip

Assume that $G'$ contains $r$ edges and $t$ isolated vertices. For convenience, let $C$ be the set of dominating vertices in $\mathcal{Q}_1$ and we define
$$		f_s(|\mathcal{Q}_1|) =(2r+t){|\mathcal{Q}_1|\choose s-1} + r{|\mathcal{Q}_1|\choose s-2}$$
and
$$	g_s (|\mathcal{Q}_1|) =  {|\mathcal{Q}_1|\choose s}+2|\mathcal{Q}_1|{|\mathcal{Q}_1|\choose s-1}+|\mathcal{Q}_1| {|\mathcal{Q}_1| \choose s-2}.$$
\begin{lemma}\label{new1}
$G[V(\mathcal{Q}_1) \cup V(G')]$ has at most $f_s(|\mathcal{Q}_1|) + g_s(|\mathcal{Q}_1|)$ $s$-cliques for $3 \leqslant s \leqslant k+1$. This maximum occurs only if $L_1 = \emptyset$ and $G[V(\mathcal{Q}_1) \cup V(G')]=K_{|C|}+M_{\ell}$, where $\ell=|V(\mathcal{Q}_1)| + |V(G')| - |C|$.
%$G[C \cup u]$ is a complete graph for all $u \in V(G') \cup L$.
\end{lemma}

\noindent {\bf Proof.}
 By Lemma \ref{lem-in}, each $K_s$ in $G[V(\mathcal{Q}_1)]$ has at most two vertices in $L$. Thus, the set of $s$-cliques in $G[V(\mathcal{Q}_1) \cup V(G')]$ can be partitioned into four subsets.

\medskip

\noindent{\bf Subset 1.} $s$-cliques spanning vertices from  both $G'$ and $G[V(\mathcal{Q}_1)]$.

By Lemma \ref{cor-1},
$|C|=|\mathcal{Q}_1|$.
For each special vertex $u$, by Lemmas \ref{cor-1}(b) and \ref{lem-special}, $u$ is adjacent to precisely one residual vertex $a\in L_1$ (without loss of generality, assume that $a\in Q_i$) and dominating vertices of some $Q_j$'s with $j\neq i$.
Hence, there are at most ${1+|C|-1 \choose s-1}= {|C|\choose s-1}$ $s$-cliques containing $u$. 
For each $u'\in V(G')$ that is not a special vertex, the neighbors of $u'$ in $V(\mathcal{Q}_1)$ form a subset of $C$. 
Hence, there are at most 
\begin{align}\label{out-eq}
	|G'|{|C|\choose s-1} + r{|C|\choose s-2} = (2r+t){|\mathcal{Q}_1|\choose s-1} + r{|\mathcal{Q}_1|\choose s-2}=f_s(|\mathcal{Q}_1|)
\end{align} $s$-cliques. Specifically, if there is no special vertex in $G'$, then the equality holds only if $G[C\cup \{u\}]$ is a complete graph for all $u\in V(G')$. 
\medskip

\noindent{\bf Subset 2.} $s$-cliques in $G[C]$.

Obviously, the number of such $s$-cliques is at most ${|C|\choose s}={|\mathcal{Q}_1|\choose s}$.
\medskip

\noindent{\bf Subset 3.} $s$-cliques in $G[V(\mathcal{Q}_1)]$ with exactly one edge from $G[L]$.

Let $abca$ be a triangle with $ab$ in $G[L]$. By Lemma \ref{lem-in}, $c$ is the  dominating  vertex of some $Q_i \in\mathcal{Q}_1$. Hence, there are at most \begin{align}\label{in-eq} e(G[L]){|C|\choose s-2}\leqslant \frac{|L|}{2}{|\mathcal{Q}_1|\choose s-2}= |\mathcal{Q}_1|{|\mathcal{Q}_1|\choose s-2} \end{align}
$s$-cliques in $G[V(\mathcal{Q}_1)]$ with exactly one edge from $G[L]$. 
Moreover, equality in Ineq. (\ref{in-eq}) holds if and only if $L_1 = \emptyset$ and $G[C \cup \{a\}]$ is a complete graph for all $a \in L$.

The sufficiency is clearly established.
Obviously, when the equality in Ineq. (\ref{in-eq}) holds, $G[C \cup {a}]$ is a complete graph for all $a \in L$ and $G[L] = M_{|L|}$. If $L_1\neq\emptyset$, then there exists a $P_3$ in $\mathcal{Q}_1$ (let us denote it as $Q_i$) such that $G[V(Q_i)]$ is a $P_3$.
Let $u, v \in L_1 \cap V(Q_i)$. If $y_i \in \{u, v\}$, since $G[C \cup {a}]$ is a complete graph for all $a \in L$, we have $G[V(Q_i)] = K_3$, which contradicts $u, v \in L_1$. Therefore, $y_i \notin \{u, v\}$, i.e., $G[V(Q_i)]=u y_i v$. Hence,  the dominating vertex of $Q_{\ell} \in \mathcal{Q}_1$ is $y_{\ell}$. 
Suppose $uw \in G[L]$. If $w \in L_2$, then by the definition of $L_2$, $w$ must have at least two neighbors in $L$, which contradicts Lemma \ref{lem-in}. Hence, $w\in (V(Q_j) \cap L_1)$ and $i \neq j$.
Since $Q_j \in \mathcal{Q}_1$, by Corollary \ref{cor-11}, $y_j$ has at least three neighbors in $G'$. It follows that there exist distinct vertices $a', b'$ in $G'$ such that $a'y_j, b'y_j \in E(G)$. Since Ineq. (\ref{in-eq}) implies that $G[C \cup {a}]$ is a complete graph for all $a \in L$, the vertices $u, w, y_i$ form a triangle. Let $V' = V(Q_i) \cup V(Q_j) \cup \{a', b'\}$. It is evident that $G[V']$ contains a $K_3 \cup P_3$, contradicting our selection of $\mathcal{Q}$.
Therefore, $L_1\neq\emptyset$.
\medskip

\medskip
%$s$-cliques in $G[V(\mathcal{Q}_1)]$ with exactly one edge from $G[L_1]$.

%Given a triangle $abca$ in $G$, where $ab$ is an edge in $G[L_1]$. By Lemma \ref{lem-in}(b), $c$ is the dominating vertex of some $Q_i\in \mathcal{Q}_2$, and $G[V(Q_i)]$ is a triangle. The number of such vertices $c$ is at most $|L_2|/2$. Hence, there are at most
%$e(G[L_1]){|L_2|/2 \choose s-2} \leqslant \frac{|L_1|}{2}{|L_2|/2 \choose s-2}$$ $s$-cliques in $G[V(\mathcal{Q}_1)]$ with exactly one edge from $G[L_1]$.\medskip

%\noindent{\bf Subset 4.} $s$-cliques in $G[V(\mathcal{Q}_1)]$ with precisely one edge from $G[L_2]$.

%Let $abca$ be a triangle with $ab$ in $G[L_2]$. By Lemma \ref{lem-in}, $c$ is the  dominating  vertex of some $Q_i \in\mathcal{Q}_1$. Hence, there are at most
%$$e(G[L_2]){|C|\choose s-2}\leqslant (|L_2|/2){|\mathcal{Q}_1|\choose s-2}$$ $s$-cliques in $G[V(\mathcal{Q}_1)]$ with exactly one edge from $G[L_2]$. \medskip

\noindent{\bf Subset 4.} $s$-cliques in $G[V(\mathcal{Q}_1)]$ with exactly one vertex in $G[L]$.

It is obvious that this number is at most 
$$
	|L|{|C|\choose s-1}=2|\mathcal{Q}_1|{|\mathcal{Q}_1|\choose s-1}.
$$

In summary, the total number of $s$-cliques in $G[V(\mathcal{Q}_1) \cup V(G')]$ is bounded from above by
$$\quad 	f_s(|\mathcal{Q}_1|)+{|\mathcal{Q}_1|\choose s} + |\mathcal{Q}_1| {|\mathcal{Q}_1| \choose s-2} +
		2|\mathcal{Q}_1|{|\mathcal{Q}_1|\choose s-1}=f_s(|\mathcal{Q}_1|) +g_s(|\mathcal{Q}_1|).$$
And the upper bound can only be achieved if $L_1=\emptyset$, Eq. (\ref{out-eq}) holds and Ineq. (\ref{in-eq}) holds.
However, $L_1=\emptyset$ implies that there is no special vertex in $G'$ by Lemma \ref{cor-1}(b).
\hfill$\square$ 

\medskip

For convenience, we define
$$f(k,s)=\max\left\{{3k-3\choose x}-{|V(\mathcal{Q}_1)|\choose x}:x\in\{s-2,s-1,s\}\right\}$$
with integers $k,s$.

\begin{lemma}\label{new2}
There are at most $h_s(|\mathcal{Q}_2|)=(1+9|\mathcal{Q}_2|)f(k,s)$ $s$-cliques contain vertices of $ V(\mathcal{Q}_2)$  from $3\leqslant s\leqslant k+1$.
\end{lemma}

\noindent {\bf Proof.}
Recall that $\mathcal{Q}_2 \subseteq \mathcal{Q}$ consists of $Q_i$ incident with at most three components of $G'$. We categorize the $s$-cliques containing vertices from $V(\mathcal{Q}_2)$ into two subsets.
\medskip

\noindent{\bf Subset A.} $s$-cliques that are either fully in $G[V(\mathcal{Q}_2)]$ or consist of vertices from $V(\mathcal{Q}_2)$ and $V(\mathcal{Q}_1)$.

Since $V(\mathcal{Q}) = V(\mathcal{Q}_1) \cup V(\mathcal{Q}_2)$ and $|V(\mathcal{Q})| = 3k - 3$, the number of such $s$-cliques is obviously at most
$${3k-3 \choose s} - {|V(\mathcal{Q}_1)| \choose s}.$$
\medskip
\noindent{\bf Subset B.} $s$-cliques include vertices from both $V(G')$ and $V(\mathcal{Q}_2)$.

Denote by $G''$ the subgraph of $G'$ consisting of components with neighbors in $\mathcal{Q}_2$.
For each vertex $w\in V(G'')$, there are at most 
$${3k-3\choose s-1}-{|V(\mathcal{Q}_1)|\choose s-1}$$ 
$s$-cliques containing $w$ and vertices from $\mathcal{Q}_2$. For each edge $f$ of $G''$, there are at most 
$${3k-3\choose s-2}-{|V(\mathcal{Q}_1)|\choose s-2}$$ 
$s$-cliques containing $f_i$ and vertices from $\mathcal{Q}_2$. Since each $Q_i \in \mathcal{Q}_2$ is incident with at most 3 components of $G'$, $G''$ has at most $3|\mathcal{Q}_2|$ components. As $G'$ is $P_3$-free, we can assume $G''$ has $a$ edges and $b$ isolated vertices, with $a \leqslant a+b \leqslant 3|\mathcal{Q}_2|$.
It follows that there are at most
\begin{align*}
	(2a+b)\left[{3k-3\choose s-1}-{|V(\mathcal{Q}_1)|\choose s-1}\right]+a\left[{3k-3\choose s-2}-{|V(\mathcal{Q}_1)|\choose s-2}\right]
	\leqslant 9|\mathcal{Q}_2| f(k,s)
\end{align*}
$s$-cliques consist of vertices from both $V(G')$ and $V(\mathcal{Q}_2)$.

Above all, the number of $s$-cliques containing at least one vertex from $V(\mathcal{Q}_2)$ is at most
\begin{align*}
	{3k-3\choose s}-{|V(\mathcal{Q}_1)|\choose s}+9|\mathcal{Q}_2| f(k,s)\leqslant (1+9|\mathcal{Q}_2|)f(k,s)=h_s(|\mathcal{Q}_2|).
\end{align*}
This completes the proof.
\hfill$\square$ 
\medskip

By using Lemmas \ref{new1} and \ref{new2}, we obtain the proof of Theorem \ref{thm}.
\medskip

\noindent {\bf Proof of Theorem \ref{thm}.}
Let $G$ be an $n$-vertex $k  P_3$-free graph  with the  maximum number of $s$-cliques, where $k\geqslant s\geqslant 3$
and $n\geqslant\max \{ g(k,s), 3k\}$.
If $G$ contains no $(k-1)P_3$, we add edges to $G$ until the resulting graph $G_1$ is a $k$-$P_3$-free graph that contains a $(k-1)P_3$.
Clearly, $G_1$ has at least as many $s$-cliques as $G$. Thus, we assume $G$ contains $(k-1)P_3$.
Let $V(\mathcal{Q})=V(\mathcal{Q}_1)\cup V(\mathcal{Q}_2)$ and $G'$ be as defined at the start of this section.

  The $s$-cliques in $G$ can be categorized into two subsets:
\begin{itemize}
  \item $s$-cliques within $G[V(\mathcal{Q}_1) \cup V(G')]$. According to Lemma \ref{new1}, $G[V(\mathcal{Q}_1)\cup V(G')]$ contains no more than $f_s(|\mathcal{Q}_1|) + g_s(|\mathcal{Q}_1|)$ $s$-cliques for $3 \leqslant s \leqslant k+1$.
  \item $s$-cliques that includes at least one vertices from $\mathcal{Q}_2$.
By Lemma \ref{new2}, there are at most $h_s(|\mathcal{Q}_2|)=(1+9|\mathcal{Q}_2|)f(k,s)$ such $s$-cliques, where $3\leqslant s\leqslant k+1$.
\end{itemize}
Therefore, the number of $s$-cliques in $G$ is at most
\begin{align*}
	&\quad f_s(|\mathcal{Q}_1|)+g_s(|\mathcal{Q}_1|)+h_s(|\mathcal{Q}_2|)\\&={|\mathcal{Q}_1|\choose s}+(2|\mathcal{Q}_1| +2r +t){|\mathcal{Q}_1|\choose s-1}+(|\mathcal{Q}_1| +r) {|\mathcal{Q}_1| \choose s-2}+(1+9|\mathcal{Q}_2|)f(k,s),
\end{align*}
where $r$ and $t$ are the numbers of edges and isolated vertices of $G'$,  respectively. Note that $|\mathcal{Q}_1| + |\mathcal{Q}_2| = k-1$ and $2r + t = n - 3k + 3$. If $|\mathcal{Q}_2| = 0$, then $f(k,s) = 0$. Additionally, observe that $r \leqslant \left\lfloor \frac{n - 3k + 3}{2} \right\rfloor$. We conclude that
\begin{align*}
	&\quad f_s(|\mathcal{Q}_1|)+g_s(|\mathcal{Q}_1|)+h_s(|\mathcal{Q}_2|)\\&= {k-1\choose s} + (n-k+1){k-1\choose s-1}+ (k+r-1){k-1\choose s-2}\\
	&\leqslant \binom{k-1}{s}+(n-k+1)\binom{k-1}{s-1} + \left\lfloor\frac{n-k+1}{2}\right\rfloor \binom{k-1}{s-2}\\
 &=f(n,k,s).
\end{align*}
%Let $N=\binom{k-1}{s}+(n-k+1)\binom{k-1}{s-1} + \left\lfloor\frac{n-k+1}{2}\right\rfloor \binom{k-1}{s-2}$.
Given that $K_{k-1} + M_{n-k+1}$ is an $n$-vertex $k P_3$-free graph containing $f(n,k,s)$ $s$-cliques, and $G$ maximizes the number of $s$-cliques among all such graphs, it follows that
$f(n,k,s) = f_s(|\mathcal{Q}_1|) + g_s(|\mathcal{Q}_1|) + h_s(|\mathcal{Q}_2|)$ and $r = \left\lfloor \frac{n - 3k + 3}{2} \right\rfloor$.
Furthermore, the equality in (\ref{in-eq}) holds. By Theorem \ref{new1}, $G$ is isomorphic to $K_{k-1} + M_{n-k+1}$.

If $|\mathcal{Q}_2|\geqslant 1$, then $|\mathcal{Q}_1|\leqslant k-2$.
Since $|\mathcal{Q}_2|\leqslant k-1$, it follows that 
\begin{align*}
	&\quad f_s(|\mathcal{Q}_1|)+g_s(|\mathcal{Q}_1|)+h_s(|\mathcal{Q}_2|)\\&\leqslant {k-2\choose s}+(n-k-1){k-2\choose s-1}+(k+r-2){k-2\choose s-2} +(1+9|\mathcal{Q}_2|)f(k,s)\\
	&\leqslant {k-2\choose s}+(n-k-1){k-2\choose s-1}+(k+r-2){k-2\choose s-2} + (9k-8)\\
	&\qquad \times\max\left\{{3k-3\choose x}:x\in\{s-2,s-1,s\}\right\}.
\end{align*}
Hence,
\begin{align*}
	&\quad f(n,k,s)-(f_s(|\mathcal{Q}_1|)+g_s(|\mathcal{Q}_1|)+h_s(|\mathcal{Q}_2|))\\
 &> (n-k-1)\left[{k-1\choose s-1}-{k-2\choose s-1} \right]- (9k-8)\\
	& \qquad \times \max\left\{{3k-3\choose x}:x\in\{s-2,s-1,s\}\right\}\\
	&\geqslant (n-k-1){k-2\choose s-2}-(9k-8)\max\left\{{3k-3\choose x}:x\in\{s-2,s-1,s\}\right\}\\
	&\geqslant0,
\end{align*}
since $n\geqslant g(k,s)={k-2\choose s-2}^{-1} \max \left\{{3k-3\choose x}:x\in \{s-2,s-1,s\}\right\} (9k-8)+k+1.$
This implies that $e(G)<f(n,k,s)$, which contradicts the choice of $G$.
\hfill$\square$

\medskip

Finally, we present the proof of Theorem \ref{theorem3.1}. 
\medskip

\noindent {\bf Proof of Theorem \ref{theorem3.1}.}
Let  
$n\geqslant 6\binom{3k-1}{k}+k$ and $s= k+1$. At this point, $f(n,k,k+1)> {{3k-1}\choose k+1}$.
We proceed by induction on $k$. The case $k=2$ is covered by Theorem \ref{theorem3.4}.
Let $k\geqslant 3$. Consider an $n$-vertex $kP_3$-free graph $H$ with a maximum number of $s$-cliques, subject to which, let $H$ contain the largest number of edges. Therefore, $H$ contains a graph $H_1=(k-1)P_3$ as subgraph. Suppose that $\mathcal{N}(H,K_{k+1}) > f(n,k,k+1)$.
For any copy of $P_3$ in $H$, since $H-P_3$ forms a $(k-1)P_3$-free graph with at least $6\binom{3(k-1)-1}{k-1} + k-1$ vertices, by induction on $k$, the number of $(k+1)$-cliques incident with the graph $P_3$ must be at least
\begin{align}
	&\mathcal{N}(H,K_{k+1})-\mbox{ex}(n-3, K_{k+1}, (k-1)  P_3) \nonumber\\
\geqslant &f(n,k,k+1)-f(n-3,k-1,k+1) +1\nonumber\\
%	&= \binom{k-2}{s-1}+ (n-k-1)\binom{k-2}{s-2}+\lfloor \dfrac{n-k-1}{2}\rfloor  \nonumber\\
%	&\quad \times \binom{k-2}{s-3}+2\binom{k-1}{s-1}+\binom{k-1}{s-2} +1 \nonumber\\
	=&\left\lfloor (n-k-1)/2\right\rfloor +2. \nonumber
\end{align}
Hence, each $P_3$ in $H_1$ must contain a vertex share with at least $(\lfloor (n-k-1)/{2}\rfloor +2)/3$ $(k+1)$-cliques.
Taking such a vertex from each $P_3$ gives us a set $U$ of $k - 1$ vertices.

Assume that $H- U$ contains a copy of $P_3$, say $Q$. 
Since $n\geqslant  6\binom{3k-1}{k}+k$ and $k \geqslant 3$, we have 
$(\lfloor (n-k-1)/{2}\rfloor +2)/3 \geqslant \binom{3k-1}{k}$.
Thus each vertex in U has degree at least $3k-1$. Therefore, we can easily obtain a $(k-1)P_3$ within $H[N[U]\setminus V(Q)]$. In this case, $H$ contains $kP_3$, a contradiction.
Therefore, $H- U$ consists of independent edges and isolated vertices. 
Since $H$ has the maximum number of edges under certain constraints, we obtain that $H= K_{k-1}+ M_{n-k+1}$, where $K_{k-1} = H[U]$ and $M_{n-k+1}=H-U$.
The proof is complete.
\hfill$\square$ \medskip

	\section{Acknowledgements}
Zhipeng Gao is supported by the Fundamental Research Funds for the Central Universities(No. XISJ24049). 
Ping Li is supported by the National Natural Science Foundation of China (No. 12201375). 
Changhong Lu is supported by the National Natural Science
Foundation of China (No. 11871222, 11901554) and Science and Technology Commission
of Shanghai Municipality (No. 18dz2271000).

\bibliographystyle{unsrt} % We choose the "plain" reference style
\bibliography{20241108The_maximum_number_of_cliques_in_k.P_3_-free_graphs}

\end{document}